\UseRawInputEncoding
\documentclass[12 pt]{article}
\usepackage{amsmath}
\usepackage{amssymb}
\usepackage{graphicx}
\usepackage{amsthm}
\usepackage[all]{xy}
\usepackage{mathtools}

\usepackage{tikz}
\usetikzlibrary{cd,arrows,shapes,automata,backgrounds,petri}
\usepackage{tikz-cd}

\usepackage{pgfplots}
 \usetikzlibrary{math}

\usepackage{enumerate}

\newcommand{\eps}{\varepsilon}
\newcommand{\s}{\mathbb{S}}
\newcommand{\SG}{\s}
\newcommand{\set}{\mathbf{Set_{3}}}
\newcommand{\ms}{\mathbf{Met_3}^{Sh}}
\newcommand{\mc}{\mathbf{Met_3}^{C}}
\newcommand{\ml}{\mathbf{Met_3}^{L}}
\newcommand{\nn}{\mathbb{N}}

\newcommand{\mtimes}{M\otimes -}

\newcommand{\rem}[1]{\relax}
\renewcommand{\o}{\circ}
\newcommand{\CC}{\mathcal{C}}
\newcommand{\DD}{\mathcal{D}}
\newcommand{\BiP}{\mathbf{Set}_2}
\newcommand{\BiMS}{\mathbf{Met}_2}
\newcommand{\Met}{\mathbf{Met}}
\newtheorem{theorem}{Theorem}

\newtheorem{lemma}[theorem]{Lemma}
\newtheorem{proposition}[theorem]{Proposition}
\newtheorem{corollary}[theorem]{Corollary}

\theoremstyle{definition}
\newtheorem{definition}[theorem]{Definition}
\newtheorem{example}[theorem]{Example}
\newtheorem{remark}[theorem]{Remark}
\numberwithin{theorem}{section}

\title{Presenting the Sierpinski Gasket in Various Categories of Metric Spaces}
\author{\parbox{\linewidth}{\centering Jayampathy Ratnayake, Annanthakrishna Manokaran, Romaine Jayewardene, Victoria Noquez, and Lawrence S.~Moss\footnote{Supported by grant \#586136 from the Simons Foundation.}}}

\begin{document}

\maketitle

\begin{abstract}
This paper studies presentations of 
the Sierpinski gasket as a final coalgebra
for functors on several categories of metric spaces
with additional designated points.
The three categories which we study differ on their morphisms:
one uses short (non-expanding) maps,
a second uses Lipschitz maps,
and a third uses continuous maps.
The functor in all cases is very similar to what we find
in the standard presentation of the gasket as an attractor.
We prove that the Sierpinski gasket itself is the
final coalgebra of a naturally-occurring functor
in the continuous setting.
In the short setting, the final coalgebra exists but 
it is better described as the completion of the initial algebra,
and this is not isomorphic to the 
Sierpinski gasket.
In the Lipschitz setting, the final coalgebra does not exist.
We determine the initial algebras in all three settings as well.
\end{abstract}

\section{Introduction}

Important mathematical objects ought to have presentations which make use of universal properties.
This idea comes through most clearly for a very prosaic structure, $(N,0,s)$, where $N$ is the natural numbers,
and $s\colon N\to N$ is the successor function.   This structure may be characterized in terms of the class $\DD$ of
triples $(X,x,f)$ first considered by Dedekind.
Elements of $\DD$ consist of a set $X$, an element $x\in X$ (so that $X$ must be non-empty), and 
a function $f\colon X\to X$.   For the 
characterization, we consider morphisms between elements of $\DD$.  
From $(X,x,f)$ to $(Y,y,g)$, an appropriate morphism would be a map $\phi\colon X\to Y$ such that $\phi(x) = y$,
and $\phi(f(z)) = g(\phi(z))$ for all $z\in X$.
Then the characterization of $(N,0,s)$ is that it is \emph{initial} in $\DD$: 
for all $(X,x,f)$, there is a unique morphism $\phi\colon (N,0,s) \to (X,x,f)$.  Moreover, up to isomorphism, $(N,0,s)$ is the unique
element of $\DD$ with this property.    This existence/uniqueness assertion underlines the principle of 
definition by recursion on the natural numbers.   It is also connected to the principle of induction on $N$
and hence to the foundations of mathematics as a whole.

Given the importance of real numbers, it is surprising that a parallel characterization of $\mathbb{R}$ or subsets of it
was slow in coming.   There were early characterizations of $\mathbb{R}$ as a complete ordered field, the earliest of 
which might be \cite{Huntington03}.   What we have in mind in this paper is a characterization less ``internal'' and more 
``external'', a characterization using mappings between objects in some larger class.  This is what 
we find with objects presented via universal properties in category theory.  It is also what we saw above with  $(N,0,s)$.
To our knowledge, the first result of the kind we are after is due to~\cite{freyd:08}.
It isolates the unit interval $[0,1]$ with its two distinguished points $0$ and $1$ as the \emph{final coalgebra} of a
certain functor on a certain category.   We shall review the appropriate definitions and results in Section~\ref{section-preliminaries} below.
We do want to mention that in results of this type it makes a lot of difference what additional structure one wants to characterize.
For example, we are not aware of characterizations of $([0,1],0,1, \times)$, where $\times$ is the usual multiplication operation.
 
Freyd's paper led in several directions.   Freyd himself went on to present an algebraic approach to analysis in~\cite{freyd:08}.
Leinster generalized  Freyd's work to provide a general theory of self-similarity, aiming at the a deeper understanding of 
fractal subsets of real or complex spaces.  The work in this paper belongs to a set of papers~\cite{HasuoJacobsNiqui10,mrrTACL,Bhat,NoquezMoss} that 
add metric information to Freyd's characterization and go on to consider \emph{metric characterizations} of other fractals, and also are influenced by Leinster's work.
When one considers categories of metric spaces, it is important to settle on the choice of morphisms.

\begin{figure}[t]
\[
\begin{array}{|l||l|l|}
\hline
\mbox{category} &  \mbox{initial algebra} &\mbox{final coalgebra}\\
\hline\hline
\set \phantom{X^X}  & (G,g) & (S,s)\mbox{, also $(\SG,\sigma) =$  the }\\
    \phantom{XX}       &   & 
         \mbox{Sierpinski Gasket as a subset of $\mathbb{R}^2$}  \\
\hline
\ms  \phantom{X^X} & (G,g) & (S,s)\mbox{, where $S$ is the completion of $G$}     \\
\hline
\hline
\ml  \phantom{X^X} &  (G_\rho,g)  &     \mbox{none exists} \\
\hline
\mc \phantom{X^X} &  (G_\rho,g)  &  (S,s)\mbox{ and $(\SG,\sigma)$}\\
    \hline
\end{array}
\]
\caption{Results on initial algebras and final coalgebras of the functor $F$ in this paper.
The new results are the bottom two rows.
In them,
$G_{\rho}$ is the 
set $G$ with the discrete metric.
\label{fig-results}}
\end{figure}

The main results in this paper concern the characterizations of the Sierpinski gasket in metric terms found in~\cite{mrrTACL,Bhat}.
A \emph{tripointed metric space}
$(X, d)$ is a tripointed set  $(X,T,L,R)$
equipped with a $1$-bounded metric $d$ (i.e.,  $d(x,y) \leq 1 \,\,\, \forall\, x,y\in X$), such that the distance between any pair of distinguished elements is $1$. 
In the papers mentioned above, the morphisms were taken to be maps
preserving this tripointed structure which additionally were 
short (non-increasing):  $d(f(x), f(y)) \leq d(x,y)$.    
In addition, the morphisms must preserve the distinguished
elements $T$, $L$, and $R$.  
In this paper, we investigate the characterization result when we vary the morphisms, allowing them to be Lipschtiz maps or continuous maps.

The main results in the paper are presented in
Figure~\ref{fig-results}.
They are stated in terms of three metric categories with a metric structure: $\ms$, $\ml$, and $\mc$.
Note that we have a chain of subcategories.
\[\ms  \subset \ml \subset \mc.\]
 There are forgetful functors from these categories ($\ms$, $\ml$ and $\mc$) to $\set$.
We shall discuss three endofunctors $F\colon \CC\to \CC$, where $\CC$ is one of our categories;
we use the same notation $F$ because in all of our settings, $F$ acts the same way on objects.
We study the initial algebra of $F$ and also the final coalgebra of $F$ in all settings. 
The first two rows of the chart were established in~\cite{mrrTACL,Bhat}, and the bottom two rows are new here.

Since we are mainly interested in fractal sets characterized as final coalgebras,
the reader might wonder why we also mention initial algebras.   The reason is that the
results in~\cite{mrrTACL,Bhat} construct the final coalgebra as the Cauchy completion of
the initial algebra (with the inverse structure).   Incidentally, this is an usual occurrence:
one can find examples where the initial algebra and final coalgebra 
are the same object with inverse structures, but what we have in this area seems rare.
In any case, this motivates our interest in the initial algebras in this paper.

\paragraph{Contribution:}
We prove that in the setting of $\mc$, the final coalgebra is the completion of the initial algebra
of $F\colon \ms\to\ms$ (with the inverse structure).
This is perhaps our main result.
We also give a new proof of this previously-shown result for $\ms$:
the final coalgebra of $F\colon \ms\to\ms$ is the completion of the initial algebra
of the same functor.  
Since the method of proof in this
paper gives both results, we feel it has an advantage over what was done previously.
We also determine the initial algebra of $F$ on $\mc$ and $\ml$: it is the initial algebra for
$\ms$ but with the discrete metric. 
We show that $F\colon\ml\to\ml$ has no final coalgebra.

\section{Preliminaries}
\label{section-preliminaries}

\subsection{Tripointed sets and spaces}

A \emph{tripointed set} is a set $X$ together with three 
distinguished points called $T$ (\emph{top}),
$L$ (\emph{left}), and $R$ (\emph{right}).
We require that $T$, $L$, and $R$ be distinct.
When it is clear from the context, we often omit the distinguished points from the description of the set.
The category $\mathbf{Set_{3}}$ 
has as objects the tripointed sets and
as 
 morphisms the functions which preserve the distinguished points.

We use subscripts to differentiate distinguished points of two tripointed sets $X$ and $Y$. For example, $T_X$ and $T_Y$ for $T$ of $X$ and of $Y$ respectively.

\begin{example}[Sierpinski Gasket] \label{Sierpinski} 
The main example is of course 
 of the Sierpinski gasket
(or triangle), denoted
$\s$.  It is defined in terms of
a system of iterated functions as in~\cite{Hut}.
 Consider the maps $\sigma_a, \sigma_b,\sigma_c:\mathbb{R}^2\rightarrow \mathbb{R}^2$ given by
\begin{equation}
\label{sigmas}
\begin{array}{lcl}
    \sigma_a(x,y)&= &(x/2, y/2)+(1/4, \sqrt{3}/4)\\
    \sigma_b(x,y)&=& (x/2, y/2)\\
    \sigma_c(x,y)&= & (x/2, y/2)+(1/2,0)
  \end{array}
    \end{equation}
Then, $\s$ may be defined as the unique non-empty compact set with  $\s =
\sigma_a(\s) \cup \sigma_b(\s) \cup\sigma_c(\s)$.

We consider $\s$ as an object of $\set$ by:
$T_{\s} = (1/2, \sqrt{3}/2)$, 
$L_{\s} = (0,0)$, 
$R_{\s} = (1,0)$.
These are the fixed points of
the contractions $\sigma_a$, $\sigma_b$, and $\sigma_c$, respectively
\end{example}

\subsection{The functor $F = M\otimes -$}

We next present the functor that is central to the discussion of this paper. Definitions below were considered previously in \cite{mrrTACL,Bhat}, and the notation follows \cite{Lein}.

We first define $F$ on $\set$. Fix the set $M=\{a,b,c\}$. 
We generally use the letter  $m$ as a variable over $M$.
Given a tripointed set $X$, $M \times X$ consists of three copies of $X$ labeled by $a,\,b$, and $c$. Let $\sim$ be the equivalence relation on  $M \times X$ generated by the relations $(b,T)\sim (a,L)$, $(a,R)\sim (c,T)$, and $(c,L) \sim (b,R)$. Let $M \otimes X$ be the set of equivalence classes of $M \times X$ with respect to $\sim$. We will denote the equivalence  class of an element $(m, x)\in M \times X$ by $m\otimes x$. The set $M\otimes X$ is lifted to a tripointed set by choosing $a\otimes T$, $b\otimes L$, and  $c\otimes R$ as $T_{M\otimes X}$, $L_{M\otimes X}$, and $R_{M\otimes X}$ respectively. We set  $FX = M \otimes X$.

This description identifies (glues) certain points of the three copies of $X$. For example the relation $(b,T)~\sim(a,L)$ identifies $L$ of copy $a$ with $T$ of copy $b$, and similarly for other relations, as shown in Figure \ref{ffig1} (see also \cite{Bhat}). Finally, the top element of the top triangle, left element of the left triangle, and right element of the right triangle were taken as the three distinguished points, making it an object of $\set$.

\begin{figure}[!ht]
\centering
\includegraphics[scale=0.5]{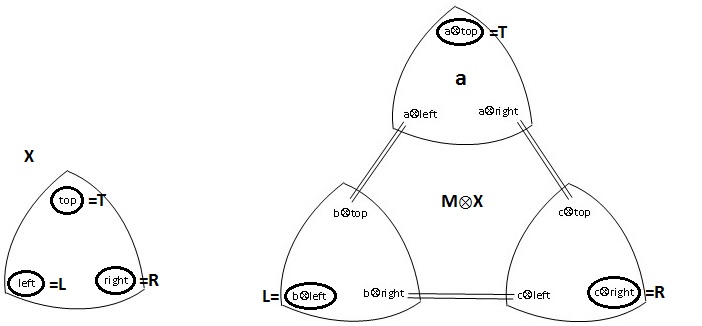}
\caption{Depiction of $FX=M\otimes X$}
\label{ffig1}
\end{figure}

The action of $F$ on morphisms is as follows. Given a morphism $f:X\rightarrow Y$ of tripointed sets, $Ff : FX=M\otimes X\rightarrow  FY=M\otimes Y$ is given by $Ff(m\otimes x)=m\otimes f(x)$. It is an easy exercise to show that $Ff$ is well-defined and preserves the distinguished elements. Hence, $Ff$ is a morphism in $\set$.

Based on the description of $F$ on objects, we frequently write
$M\otimes f$ for $Ff$. With this notation, $M\otimes -$ is an alternative notation for 
the functor $F$. The definition of $\mtimes$ can be extended
from sets to metric spaces, thereby defining endofunctors on $\ms$,   $\ml$ and $\mc$.
These  are also designated by $F=\mtimes$.
 First, for a given tripointed metric space $(X,d)$, $M\times X$ is given the metric defined as follows.
	$$d_{M\times X}\left((m,x),(n,y)\right)=\left\{%
	\begin{array}{ll}
	\frac{1}{2}d(x,y), & \hbox{$m=n$;} \\\\
	1, & \hbox{$m\neq n$.}
	\end{array}%
	\right.$$ 
Then, $M\otimes X$ is defined to be the  
\emph{quotient metric space} of  $M\times X$.  That is, we declare identified points as having distance $0$, and then take the distance between $m_1\otimes x,m_2\otimes y\in M\otimes X$ to be the infimum over the sums of the segment lengths in finite sequences of points starting with $m_1\otimes x$ and ending with $m_2\otimes y$.  Though the quotient metric in general is only a pseudometric, in our case it is indeed a metric and the distance between two elements can be computed explicitly as follows. 
	
\begin{lemma}[Lemma 2.3, \cite{mrrTACL}]
\label{LemmaMetricOnTensor}
  The metric in $M\otimes X$ is given by:
\begin{align*}
d_{M\otimes X}((a\otimes x),(a\otimes y)) &=
	\frac{1}{2}d_X(x,y) \\
d_{M\otimes X}((a\otimes x),(b\otimes y)) &=\frac{1}{2} \min \{ d_X(x,L)+d_X(T,y)~,~d_X(x,R)+1+d_X(R,y) \}
\end{align*}
More generally, $d_{M\otimes X}((m_1\otimes x),(m_2\otimes y))$ may be calculated
similarly for all $m_1,m_2\in M$ and $x,y\in X$.
\end{lemma}

Recall that the metric $d$ on a tripointed space $X$ is
required to be  $1$-bounded.
It is easy to check that  the  metrics on $M\times X$ and  $M\otimes X$ are also $1$-bounded.
 Moreover, it can be easily verified that the distance between the distinguished elements of $M\otimes X$ is $1$. 
 Hence $(M\otimes X,d_{M\otimes X})$  is an object of the category of tripointed metric spaces.

To define $F=M \otimes -$ on the categories $\mc$, $\ml$, and $\ms$, 
we use the following result:

\begin{lemma}
\label{presFunctions}
Let $X$ and $Y$ be two tripointed metric spaces. If $f: X \rightarrow Y$  has any of the following properties then so does $Ff=M\otimes f$.
\begin{enumerate}[i).]
\item \label{cnt} Continuous
\item \label{lip} Lipschitz
\item \label{shrt} Short map
\item \label{isoemb} Isometric embedding
\end{enumerate}
\end{lemma}

\begin{proof}
(\ref{cnt}). 
Given $f$ is continuous, we shall first show that $M\otimes f$ is continuous at $a\otimes x\in M\otimes X$. Fix $a \otimes x \in M\otimes X$ and $\epsilon>0$. Since $f$ is continuous, $\exists\,\delta_0 >0$ such that $d(f(x),f(y))<\epsilon $,  whenever $d(x,y  )<\delta_0 $.  Choose  $\delta=\min \{   
\frac {\delta_0}{2},\frac {1}{4} \}$. Suppose $d_{M\otimes X}(a\otimes x, b\otimes y)<\delta$. Then we have $d_{X}(x,L_{X})< \delta_0$ and $d_{X}(y,T_{X})<\delta_0 $. By the continuity of $f$ we have $d_{Y}(f(x),f(L_{X}))<\epsilon$ and $d_{Y}(f(y),f(T_{X}))<\epsilon$. Thus, $d_{M\otimes Y}(Ff(a\otimes x), Ff (b\otimes y) )<\epsilon$, which is the required condition for $F(f)$ to be continuous at $a\otimes x$. 

One can show the continuity of $M\otimes f$ for the other cases in a similar way. Thus, we have that $M\otimes f$ is continuous when $f$ is continuous.\\
(\ref{lip}).	Let $~f~$ be a Lipschitz continuous function with Lipschitz constant $k\geq 1$. We shall note that 
\begin{align*}
\text{min} & \left\{ d\left(f(x),L\right)+d\left(T,f(y)\right)~,~d\left(f(x),R)+1+d(R,f(y)\right) \right\} \\ 
   & =\ \min \left\{ d(f(x),f(L))+d(f(T),f(y))~,~d(f(x),f(R))+1+d(f(R),f(y)) \right\}\\
   & \leq  \min \left\{k\cdot d(x,L)+k\cdot d(T,y)~,~k\cdot d(x,R)+k\cdot 1+k\cdot d(R,y) \right\}\\
   & =  k \cdot \min \left\{d(x,L)+d(T,y)~,~d(x,R)+1+d(R,y)\right\}
\end{align*}

\noindent The other cases for the distance (as in Lemma \ref{LemmaMetricOnTensor}) are even easier. Thus, 
\[
d_{M\otimes Y} \left(Ff(a\otimes x),Ff(b\otimes y)\right) = d_{M\otimes Y} \left(a\otimes f(x), b\otimes f(y)\right) \leq k \cdot d_{M \otimes X} \left(a\otimes x,b\otimes y\right) 
\]
This shows that $M\otimes f$ is Lipschitz whenever $f$ is Lipschitz.

(\ref{shrt}) is proved in Proposition 8 of \cite{Bhat} and the proof of (\ref{isoemb}) is similar to that of (\ref{lip}).
\end{proof}

It now follows from Lemma \ref{presFunctions} that $F$ is a well-defined endofunctor on $\mc$ and that it restricts to an endofunctor on the categories $\ms$  and  $\ml$. We use the same letter $F$ to refer to all these functors and specify the domain only when there is ambiguity.

Here is 
the example which  is central to our discussion.

\begin{example}\label{Sierpinski2}
We build on Example~\ref{Sierpinski}.
Define the map $\tau \colon M\otimes \s \rightarrow \s$ by 
\[ \tau\left(m\otimes x \right)=\sigma_m\left(x \right).\]
We check easily
that $\tau$ is a well-defined bijection preserving the distinguished elements.  Then $\s$ carries a coalgebra structure for $F$ on $\set$, given by \[\sigma = \tau^{-1} : \s \rightarrow M\otimes \s.\]
Elementary calculations show
that  $\tau(a\otimes R) = \sigma_a(R) = (\frac{3}{4}, \frac{\sqrt{3}}{4})$,
and
\[
d_{\s}(\tau(b\otimes L), \tau(a\otimes R)) = 
\frac{\sqrt{3}}{2} <  1 = d_{M\otimes \s}(b\otimes L, a\otimes R).
\]
This shows that $\tau$ has no inverse in $\ms$.  Obviously $\tau$ does have an inverse 
in $\set$.  And for that matter, it has an inverse in $\mc$ and $\ml$.
\end{example}

\begin{definition}
Given a tripointed set $X$ (or a metric space), we will denote $F^2X=F\left(FX\right) = M\otimes \left(M\otimes X \right)$ by $M^2\otimes X$. Similarly, given a morphism $f$ in $\set$, $M^2\otimes f$ will stand for
$F^2f=F\left(F f\right)=M\otimes\left(M\otimes f\right)$. Likewise, for any integer $n\in \nn$,  the $n$-fold composition of $F$ will be denoted by $M^n\otimes -$. We set $M^0\otimes X =X$ and $M^0 f=f$. A typical element of $M^n\otimes X$ will be denoted by $m_0\otimes \cdots \otimes m_{n-1} \otimes x$.
Frequently we abbreviate a long ``tensor product'' of elements of $M$ in boldface,
writing $\mathbf{m}$ for $m_0\otimes m_1 \otimes \cdots \otimes m_{n-1}$.
\end{definition}

The following observation easily follows from Lemma \ref{LemmaMetricOnTensor}.

\begin{lemma}[Lemma 15, \cite{Bhat}]
\label{totalybound}
Let $X$ be a tripointed metric space (i.e., an object of $\mc$), $n\in \nn$ and 
$m_0, \cdots,m_n\in M$.
Then for all $x,y \in X$,
\[
d_{M^n\otimes X} \left(\mathbf{m} \otimes x ~,~  \mathbf{m} \otimes y \right) \leq 2^{-n} 
\]
In particular, for all $x$, $y$, $x'$, $y'\in X$
\[
\left| d_{M^n\otimes X}\left( \mathbf{m} \otimes x~,~ \mathbf{m} \otimes y \right) - 
d_{M^n\otimes X}\left( \mathbf{m} \otimes x'~,~ \mathbf{m} \otimes y' \right) \right| \leq 2^{1-n}
\]
\end{lemma}

\begin{example}
The set $I=\{T, L, R\}$, with distinguished points as suggested by the notation, is the initial object of the category $\set$:
for every tripointed set $X$, there is a unique $\set$-morphism from $I$ to $X$.
We endow $I$ with the discrete metric, so that $d(z,z') = 1$ for $z\neq z'$.  This is then the initial object in 
 all of our metric categories.

Here is $F I$ in $\set$.  It is a set with six elements:
\[
\begin{array}{l}
a \otimes T \\
a\otimes L = b\otimes T \\
a \otimes R = c\otimes T
\end{array}
\qquad
\begin{array}{l}
 \\
b\otimes L  \\
b \otimes R = c\otimes L
\end{array}
\qquad
\begin{array}{l}
\\
 \\
c \otimes R 
\end{array}
\]
The tripointed structure is $T_{F I} = a \otimes T$,
 $L_{F I} = b \otimes L$,
  $R_{F I} = c \otimes T$.
  
We exhibit   $F I$ in $\ms$, we need only give the metric on these points.
For $z\neq z'$ in $\{ T, L, R\}$, $d_{FI}(a\otimes z, a\otimes z')=\frac{1}{2}$, and similarly for $b$ and $c$.
Further,
\[ d_{FI}(a\otimes T, b\otimes L)=1 =  d_{FI}(a\otimes T, b\otimes R).\]
$F^2 I$ is then a set with $3(6) -3=15$ elements.
These may be exhibited in the form $m_0\otimes m_1\otimes z$, where $m_0, m_1\in M = \{a,b,c\}$ 
and $z\in \{T, L, R\}$.  This appears to be $27$ elements, but recall that we
have made identifications at several steps.
The metric between any two points may again be calculated
by Lemma~\ref{LemmaMetricOnTensor}.

\end{example}

\subsection{Algebras and Coalgebras}

This paper contributes to the discussion of natural mathematical objects viewed in terms of their universal properties,
specifically as either initial algebras or final coalgebras.

Let us begin with the general definition of algebras and coalgebras.
For both definitions, one begins with a category $C$ and an endofunctor 
$F\colon C\to C$. (An endofunctor is a functor from some category to itself.)

An \emph{$F$-algebra} is a pair $(A, \alpha)$, where $\alpha\colon FA \to A$
is a morphism in the underlying category $C$.
Given algebras $(A,\alpha)$ and $(B, \beta)$, an
\emph{algebra morphism} is a morphism $h\colon A \to B$ so that 
$h\o \alpha = \beta \o Fh$.
See the square on the left just below:
\[
\begin{tikzcd}
FA  \arrow{r}{\alpha} \arrow{d}[swap]{Fh}&  A \arrow{d}{h}\\
FB  \arrow{r}[swap]{\beta} &  B
\end{tikzcd}    
\qquad\qquad
\begin{tikzcd}
A  \arrow{r}{\alpha} \arrow{d}[swap]{h}&  FA \arrow{d}{Fh}\\
B  \arrow{r}[swap]{\beta} &  FB
\end{tikzcd} 
\] 
Frequently, one elides the name of the functor and the underlying category, and just speaks of \emph{algebras}.  For example:
an algebra is \emph{initial} if it has
a unique algebra morphism to every algebra.

Dually, an 
\emph{$F$-coalgebra} is a pair $(A, \alpha)$, where $\alpha\colon A\to FA$
is a morphism in  $C$.
Given coalgebras $(A,\alpha)$ and $(B, \beta)$, a
\emph{coalgebra morphism} is a morphism $h\colon A \to B$ so that 
$\beta\o h =  Fh\o \alpha$.
See the square on the right above.
A coalgebra is \emph{final} 
if every coalgebra has a unique 
coalgebra morphism into it.

We conclude with a number of examples that set the stage for our work.
For more examples and for the theory surrounding them, see e.g.~\cite{AMM}.

\begin{example}
\label{ex:NplusOne}
On $\mathbf{Set}$, consider the functor $GX = X + 1$.
An algebra for $G$ is a pair $(A,\alpha)$, where $\alpha \colon A+1 \to A$.
This map $\alpha$ may be seen as the choice of a point $a_0\colon 1 \to A$ together with
a self-map $a_1\colon A\to A$. One example is $N$, the natural numbers together with $0$
(here $0$ is taken to be a map $n_0\colon 1 \to N$) and the successor function $a_1(n) = n+1$.
This algebra turns out to be initial; this is equivalent to the assertion that a function
with domain $N$, say $f\colon N \to X$, is determined
uniquely by an algebra structure on $X$; this again amounts to an element of $X$ and a self map on $X$.

The final coalgebra for the same functor turns out to be $N^{\infty} = N\cup\{\infty\}$
with the structure $n\colon N^{\infty}\to N^{\infty} + 1$ given by 
$n(0) = {*}$, where the summand $1$ is taken to be $\{*\}$; $n(k) = k-1$ for other natural numbers $k$;
and $n(\infty) = \infty$.  
Given a coalgebra $(A, \alpha\colon A \to A +1)$, we naturally associate a 
partial function $\widehat{f}\colon A\to A$.  Its domain is 
the set of $a\in A$ which $f$ sends to an element of $A$ (rather than to the element of $1$).
The coalgebra morphism 
$h\colon A \to N^{\infty}$ is given by taking $a\in A$ to the supremum
in $N^{\infty}$ of all $n$ such that
$\widehat{f}^{n}(a)$ is defined.  
\end{example}

\begin{example}
Consider the category $\Met$ of metric spaces whose metric is bounded by $1$
using as morphisms the short maps.  The terminal object $1$ in this category is the one-point space.
The coproduct operation $+$ takes the disjoint union of the summands,
putting distance $1$ between points in different spaces.
Let us again write $GX = X+1$, with this understanding of $1$ and $+$. 
Its initial algebra is $N$ with the discrete metric,
and the same structure $n$ (this is short because every map with a discrete domain is short).
As for the final coalgebra, it too is the final coalgebra for $X+1$ 
on $\mathbf{Set}$
which we saw in Example~\ref{ex:NplusOne}, again with the discrete metric.
\label{nmet}
\end{example}

\begin{example}
Let us change the functor in Example~\ref{nmet} to the functor $HX = \frac{1}{2}X+1$
where  $\frac{1}{2} X$ is the space $X$ with distances scaled by $\frac{1}{2}$.
Then the initial algebra is $N$ with the metric given by $d(n,m) = 2^{-\min(n,m)}$ for $n\neq m$.
The structure is again the same as in $\mathbf{Set}$.
As for the final coalgebra, it is the same set $N^{\infty}$ with the metric given as in
the initial algebra, where we take $\infty > n$ for all $n \in N$.
So the initial algebra looks like a Cauchy sequence, and the final coalgebra is its completion.
We shall see this same phenomenon in the some (but not all) of the functors in this paper.
\end{example}

\begin{example}
\label{ex-bipointed}
Freyd (in ~\cite{freyd:08}) found a characterization of the real
unit interval $[0,1]$ as a final coalgebra.
The category  $\BiP$ of \emph{bipointed sets}
has as objects tuples $(X,\bot,\top)$, where
$X$ is a set and $\bot,\top\in X$.
We require $\bot\neq\top$.
For example, $[0,1]$ is a bipointed set with $\bot = 0$ and $\top = 1$.
Morphisms in $\BiP$ are set functions which preserve $\bot$ and $\top$.
The functor $K\colon\BiP\to\BiP$ takes a
bipointed set $X$ to
\[
(((X\times\{0\}) \cup (X\times\{1\}))\setminus\{(\top,0),(\bot,1)\},
(\bot,0), (\top,1))
\]
Put differently, this is 
two copies of $X$ with $\top$ in the
first copy identified with $\bot$ in the second, and
with bipointed structure given by $\bot$ in the first copy and $\top$ in the second.
The coalgebra structure on $[0,1]$ takes 
$x$ in the first copy to $\frac{1}{2}x$, and $x$ in the second copy to
$\frac{1}{2}x + \frac{1}{2}$.
Freyd showed that this coalgebra is final.
\label{ex-freyd}
\end{example}

\begin{example}
The category $\BiMS$ of bipointed metric spaces 
combines $\Met$ and $\BiP$.
Objects are metric spaces with distinguished points $\bot$ and $\top$;
these are required to be of distance $1$.
Morphisms are short maps preserving $\bot$ and $\top$.
The same functor $K$ adapts to this setting.
Its initial algebra is the set of dyadic rational numbers in $[0,1]$,
with the usual metric and with the algebra structure
adapted from what we saw in Example~\ref{ex-freyd}.
Its final coalgebra is the same as that of Example~\ref{ex-freyd},
and again the metric in this case is the usual one.
For proofs, see~\cite{AMM,Bhat}.
\end{example}

\begin{example}
\label{ex-sigma}
Example~\ref{Sierpinski2} defines a morphism  
$\sigma \colon \s\to M\otimes \s$
in $\set$, $\mc$, and $\ml$.
Thus, we have coalgebras.
It will turn out that in $\set$ and $\mc$, this coalgebra is final.
This is not the case in $\ml$: no final coalgebra exists there.  
In $\ms$, there is a final coalgebra
but it is constructed differently, as the completion of the initial algebra.
\end{example}


\section{Previous Results}
\label{section-previous-results}

This paper extends results in the area, notably 
\cite{Bhat}.
Since we are going to use some details of 
the work in those papers, we must present some of the development.   
We review the initial algebra of $F$ and then describe a canonical metric on it and its Cauchy completion.

\subsection{Initial Algebras of $F$}
\label{section-initial-algebra-construction}

First, let us discuss the initial algebras of $FX = M\otimes X$ on all of our categories.
As is common throughout mathematics, we may construct an initial algebra by \emph{iteration in $\omega$ steps starting with the initial object}.
Here is what this means.
The \emph{initial chain} of $F$ is the sequence of objects and morphisms shown below:
\begin{equation}
\label{initialChain}
\begin{tikzcd}
I \arrow{r}{!} & 
F I \arrow{r}{F !} &
F^{2} I  
& \cdots   & 
 F^{j-1} I  \arrow{r}{F^{j-1} ! } &  
 F^{j}  I & 
 \cdots & 
\end{tikzcd}     
\end{equation}
where `$!$' is the unique morphism from $I$ to $F I$.
(Concretely, 
the fact that $\set$-morphisms preserve the distinguished points tells us that
$!(T) = a\otimes T$, $!(L) = b\otimes L$, and $!(R) = c\otimes R$.)

At this point, we need to take the \emph{colimit} of the chain (\ref{initialChain}).
Here is a concrete description of the colimit in $\set$.
It is the set $G$ whose elements are the expressions 
\begin{equation}
 \label{expressions}
 m_0\otimes m_1 \otimes \cdots m_{n-1} \otimes z,
 \end{equation}
where $n \geq 0$, $z\in \{T,L,R\}$, and $m_0,\ldots, m_{n-1}$ belong to $M = \{a,b,c\}$.
More precisely, we must make some identifications here 
\begin{equation}
\label{mustmake}
 a\otimes L = b\otimes T 
\qquad
a\otimes R = c\otimes T 
\qquad
b\otimes R = c\otimes L
\end{equation}
These identifications imply further identifications: for example, $b\otimes c\otimes a\otimes L = b\otimes c\otimes b\otimes T $.
We allow $n = 0$.   This enables us to endow this set $G$ with the tripointed structure $T_G = T$, $L_G = L$,
$R_G = R$.   
We have an algebra structure $\tau\colon M\otimes G \to G$ given by 
\[
\tau(m^*\otimes (m_0\otimes m_1 \otimes \cdots m_{n-1} \otimes z)) = m^*\otimes m_0\otimes m_1 \otimes \cdots m_{n-1} \otimes z
\]

The general category-theoretic result in this area
is due to Ad\'amek:

\begin{proposition}[\cite{Adamek74}]
\label{prop-Adamek}
Let 
$\CC$ be any category, and let
$F:\CC\to \CC$ be any endofunctor.
Assume that $\CC$ has an initial object $I$, 
and that the chain (\ref{initialChain}) has a colimit, say $A$.
Finally assume that $F$ preserves the colimit of 
(\ref{initialChain}).  Then there is a canonical morphism
$\alpha\colon FA\to A$ such that $(A,\alpha)$ is
an initial algebra of $F$.
Indeed, $\alpha$ may be characterized
as the unique morphism so that 
$\alpha\o F c_i = c_{i+1}$,
where $c_i \colon F^n I\to A$ is 
the map into the colimit.
\end{proposition}

We use this to build an initial algebra for $F$ on $\set$ and $\ms$.
We discuss the situation with $\ms$.
The initial object in  is
$I = \{T,L,R\}$.
The map $!\colon I\to FI$ is an isometric embedding.   
It follows from Lemma~\ref{presFunctions} that 
each space $F^j I$ is isometrically embedded in its successor $F^{j+1}I$.   As a general fact about metric spaces,
we may take the (disjoint) union of these spaces, and we obtain a metric space which isometrically embeds each $F^j I$.
As a set, this union is $G$ from above.
The maps $c_j\colon F^j I\to G$
are inclusions, and they preserve the distance.  
In particular, $c_0\colon I = F^0 I \to G$ is an isometric embedding, and we thus have a tripointed metric structure on $G$.

\begin{remark}\label{isometricembeddingintoG} In fact, for every $n$, $M^n\otimes I$ embeds isometrically into $G$, so for ${\bf m}\otimes z,{\bf n}\otimes z'\in M^n\otimes I$, $d_G({\bf m}\otimes z,{\bf n}\otimes z') = d_{M^n\otimes I}({\bf m}\otimes z,{\bf n}\otimes z')$.   \end{remark}

\begin{lemma} $\tau: F G \to G$ is an isometry.
\end{lemma}

So $F$ preserves the colimit
of the initial squence.
We  apply Ad\'amek's Theorem.

\begin{theorem}
\label{theorem-ms-initial}
$(G,\tau)$ is an initial algebra for $F$ on $\ms$.
\end{theorem}

\rem{
We offer new proofs of these last results 
Lemma~\ref{lemma-isshort}
and Theorem~\ref{theorem-ms-initial}
Proposition~\ref{shortOff}
and Theorem~\ref{theorem-mshort}.
}

Incidentally, these results on $(G,\tau)$
do not hold for
$(G,\tau)$
viewed as an object of $\mc$ or $\ml$:
see Proposition~\ref{prop-initial-alg-continuous}.

We close this section with a
detail on the morphisms
from $(G,\tau)$ to other algebras.

\begin{lemma}
\label{lemma-initiality}
Let $(A, \alpha\colon FA \to A)$ be any $F$-algebra on $\ms$.
Then the unique 
$F$-algebra morphism $\phi\colon (G,\tau)\to (A,\alpha)$
is given by recursion on the length of expressions in (\ref{expressions}):
$\phi(T_G) = T_A$, $\phi(L_G) = L_A$, $\phi(R_G) = R_A$, and 
\[\phi( m_0\otimes m_1 \otimes \cdots \otimes m_{n-1} \otimes z) 
= \alpha(m_0\otimes \phi (m_1 \otimes \cdots\otimes m_{n-1} \otimes z)).\]
\end{lemma}

All of the results which we stated above for $\ms$ also hold for $\set$.

\subsection{Cauchy Completion as an endofunctor on $\ms$}

Recall the operation of taking a metric space to its Cauchy completion.  
We write this as $C_0\colon \Met\to \Met$,
where $\Met$ again
is the category of $1$-bounded metric spaces and short maps.  
This operation extends to a functor on all of our categories.   
\begin{proposition}\cite[Lemma 13]{Bhat}
\label{prop-Cauchy}
$C_0$ lifts to a functor $C\colon \ms\to \ms$:
\[
\begin{tikzcd}
\ms \arrow{r}{C} \arrow{d}[swap]{U}&  \ms \arrow{d}{U}\\
\Met  \arrow{r}[swap]{C_0} &  \Met
\end{tikzcd}    
\]
Here $U$ is the forgetful functor, dropping
the tripointed structure.
Furthermore,
there is a natural transformation
$\eta\colon  C\o F \to F\o C$ whose components
are isomorphisms.
\end{proposition}

The extension of $C_0$ to $C$ is determined by the way that we give a tripointed structure to
the completion of (the underlying space of) a tripointed space:
we do this by using the constant sequences $(z,z, \ldots)$
of the distinguished points $z\in \{T, L, R\}$.

We need to know that  $C_0$ and $C$ preserve isometric isomorphisms in $\Met$ and $\ms$,
respectively.   These are 
the categorical isomorphisms in these categories, and every endofunctor preserves isomorphisms.


\begin{definition}
We write $(S,\psi\colon M\otimes S\to S)$ for the image under $C$ of the initial 
algebra.
In more detail, $S = M(G)$,
and 
$\psi = \eta_G \o g$, where 
$\eta_G \colon C(M\otimes G) \cong 
M\otimes C(G)$ is from Proposition~\ref{prop-Cauchy}.
\end{definition}

\begin{remark}
\label{remark-S}
Since $g$ is an isometric bijection and the completion functor $C$ preserves such maps,
it follows that $\psi$ is invertible.   We write $s:S\rightarrow M\otimes S$ for its inverse.
We shall need 
an explicit description of
this map $s$.
For this, let us fix some ordering $<$ of $M$, say $a < b < c$.
Given $(x_n)_n$,  we view every $x_n$ as $m_n\otimes x_n'$ for some $m_n$ (using equivalences in 
            (\ref{mustmake}) if necessary for the corners if $x_n$ just lives in $I$). 
            Choose $m$ which is the first via our ordering $<$ to appear infinitely many times, and 
            define $s((x_n)_n) = m\otimes (x_n')_n$.  
\end{remark}

\begin{theorem}
\label{finOnSet}
$(S, s)$ is the final coalgebra for $M\otimes-$ on $\ms$.
\end{theorem}

\begin{theorem} [\cite{Bhat}]
$(S, s)$ is bilipschitz isomorphic to $(\s,\sigma)$.
That is,  there is a bijection between these spaces which
is bilipschitz.   The two spaces are thus isomorphic
in $\mc$.  However, they are not isomorphic in $\ms$.
\end{theorem}

\section{Tripointed Metric Spaces with Continuous Maps}
\label{section-mc}

\rem{
\begin{itemize}
    \item $G$ isn't the initial algebra, but $G_\rho$ is.
    \item Finding the final coalgebra relies on $S$ being the completion of $G$, however this is somewhat misleading as $G$ (with the metric whose completion we're using to find the final coalgebra) is no longer the initial algebra in this setting. 
    \item The case against $\mc$: morally, we're viewing $G$ more like we do in $\set$ without using any information from or in the metric setting (like the quotient metric or anything that makes sense) to get the initial algebra. 
    \item The case for $\mc$: we do precisely get the Sierpiniski gasket itself as the final coalgebra - but this is really only for the same reason as in $\set$, so in some sense this is not the RIGHT way to view these as metric structures, if the point is to capture some metric information.
    \item A note that this approach of the explicit description of the morphism from some arbitrary coalgebra $X$ to $S$ does work to obtain the previous results of $S$ being the final coalgebra in $\set$ and $\ms$. 
    \item We still use the fact that $G$ is the colimit of the initial chain and that $S$ is the completion of it to get that $S$ is the final coalgebra, but the downside is that $G$ is NOT the initial algebra in this setting, so we lose the connection between initial and final coalgebra. €
\end{itemize}
}

\subsection{Initial $F$-Algebra}
\label{section-initial-algebra-continuous}

We reviewed the most common construction of
initial algebras for endofunctors in Section~\ref{section-initial-algebra-construction}.
Indeed, the initial algebra of $F$ on $\ms$
which we described there was obtained from
Ad\'amek's Theorem (Proposition~\ref{prop-Adamek}).
We now show that the same method works for 
$\mc$.  However, it turns out that the metric on the initial algebra in $\mc$ is discrete and therefore is uninteresting.
This somewhat disappointing result also suggests that the initial algebra of $F$ on $\mc$ will not be related to the final coalgebra in any interesting way.  
In $\ms$, we obtained the final coalgebra
of $F$ as the completion of the initial algebra.
The initial algebra in $\mc$ is already complete,
being discrete.  So in
 Section~\ref{section-final-coalgebra-mc} below,
we shall see that the final coalgebra on $\mc$
may be constructed by completing the initial algebra
on $\ms$ instead of $\mc$.

To apply Proposition~\ref{prop-Adamek},
the first step is
to take the colimit of the initial sequence
of the functor, repeated below:
\begin{equation}
\label{initialChainAgain}
\begin{tikzcd}
I \arrow{r}{!} & 
F I \arrow{r}{F !} &
F^{2} I  
& \cdots   & 
 F^{j-1} I  \arrow{r}{F^{j-1} ! } &  
 F^{j}  I & 
 \cdots & 
\end{tikzcd}     
\end{equation}
$I$ again is the three-point
discrete metric space $\{T,L,R\}$.
Each $FI$ is a finite space.  
Every finite metric space 
$(X,d)$
is homeomorphic
to the space $(X,d_{disc})$,
where $d_{disc}$ is the discrete metric on $X$.
The homeomorphism is the identity function on $X$.
This bijection is invertible in the category with continuous maps (but not in the category with short maps).

\begin{lemma}
The colimit of 
(\ref{initialChainAgain}) in $\mc$ is 
 $G_{\rho} = (G,d_{disc})$.
 Moreover, the colimit maps 
$F^n I \to G_{\rho}$ are the same in $\set$ and $\mc$.
\end{lemma}

\begin{proof}
Recall that there is an adjunction $D\dashv U$
where $\mbox{\sf Met}$ is the category of
$1$-bounded metric spaces  
and continuous functions,
$D\colon \mbox{\bf Set}\to \mbox{\bf Met}$
gives the discrete metric, and $U$ is the forgetful functor.

This fact boils down to the assertion that
every $\mbox{\bf Set}$ function $f\colon X\to Y$
gives a continuous function $Df\colon DX\to DY$;
the continuity is immediate.
All of this is without the tripointed structure, but it is easy to see that we have the same adjunction when
we add back that additional structure.
Our point in the previous paragraph shows that
(\ref{initialChainAgain}) in $\mc$
is naturally isomorphic to the image
under the functor $D$ of 
(\ref{initialChainAgain}) in $\set$.
Since $D$ is a left adjoint, it preserves colimits.
This gives our result.
\end{proof}

\begin{proposition}
The initial algebra of $F\colon \mc\to\mc$   is
$(G_{\rho},g)$.
Moreover, for every algebra
$(A,\alpha)$ in $\mc$, the algebra morphisms
$G_{\rho}\to A$ are the same in $\set$ and $\mc$.
\label{prop-initial-alg-continuous}
\end{proposition}

\begin{proof}
To apply Proposition~\ref{prop-Adamek},
we must check that $F$ preserves the colimit
of (\ref{initialChainAgain}).
This holds in $\set$, and
so again the adjunction shows that it is preserved in
$\mc$.
We also need to see that the structure map
$g\colon FG\to G$ is the same in $\set$ and $\ms$.
For this, we need a the last assertion in 
Proposition~\ref{prop-Adamek}: 
$g$ in $\set$ is the unique
function so that $g\o F c_i = c_{i+1}$,
where $c_i \colon F^n I\to G$ is the colimit map.
This characterization in $\set$ immediately
applies in $\mc$ as well.
Thus $g$ is the algebra structure in $\mc$.
The same argument shows the last assertion in
our result.
\end{proof}

\begin{remark}
Now that we know that the carrier of the
initial algebra of $F$ on $\mc$ is a discrete space,
it follows that this same carrier space is not
$G$ with its metric as initial algebra in $\ms$.
The reason is that this latter space has no isolated
points, so it cannot be homeomorphic to a discrete space.
\end{remark}

\subsection{Final $F$-Coalgebra}
\label{section-final-coalgebra-mc}

What drives this paper is the fact that in several known settings, 
the final coalgebra turns out to be the completion of the initial algebra
(with the inverse structure).  This happens with bipointed metric and short maps and Freyd's functor (\cite{freyd:real}), where the initial algebra
is the dyadic rationals with the usual metric (\cite{AMM}).
It happens as well for tripointed sets and our functor $F$ (\cite{Bhat}),
where the final coalgebra is indeed the completion of the initial algebra.
But this cannot happen with $\mc$, since the initial algebra is
discrete and hence trivially complete, and this space is easily seen
to not be the final coalgebra.

In this section we will show that 
the final coalgebra on $\mc$ is the completion of the
initial algebra of $F$ on $\ms$.
This same space was shown to be the final coalgebra of $F$ on $\ms$,
as we stated in Theorem~\ref{finOnSet}.
The work which we do here
for $\mc$ does not depend on the earlier result for $\ms$, and indeed
our work below
gives a new proof that the completion of $G$ is a 
final coalgebra on $\ms$.

Recall that we write
$(G,g)$ for the initial algebra on $\ms$, 
$S$ for the completion of $G$,
 and $s$ for the image under the
completion functor of $g^{-1}$.
So we have a coalgebra $(S,s)$ in $\ms$.
As we have seen, $s:S\rightarrow FS$ is a short map, so it is continuous.  Thus, we know $(S,s)$ is a coalgebra in $\mc$.  

Here we will give an explicit description of the required coalgebra morphism from an arbitrary coalgebra into $S$. 
Let $e: X \rightarrow M\otimes X = F(X)$ be a coalgebra for $F$ in $\mc$.
By iterating this coalgebra, we obtain the following chain.
\begin{equation}
\nonumber
\label{coalgebraChain}
\begin{tikzcd}
X \arrow{r}{e} & F X\arrow{r}{F e} &  \cdots F^{n-1}  X \arrow{r}{F^{n-1} e} & F^{n}\otimes X \arrow{r} & \cdots
\end{tikzcd}     
\end{equation}
For each $~x\in X$, we construct a sequence $\left(\chi_n \right)_{n\in \nn}$ such that 
\[ \chi_n = m_{0}\otimes m_{1}\otimes\cdots\otimes m_{n-1}  \otimes x_{n} \in M^n\otimes X
\]
by iterating $x$ along the above chain. Specifically, we let $\chi_0=x$ and define $\chi_n$ inductively by  
\[\chi_n = \ \left(M^{n-1} \otimes e \right) \left ( \chi_{n-1}\right).\]
Given $\chi_{n-1}=m_{0}\otimes m_{1}\otimes...\otimes m_{n-2}  \otimes x_{n-1}$, we have 
\[
\begin{array}{lcl}
\chi_n &  = & \ \left( M^{n-1} \otimes e \right) \left(m_{0}\otimes m_{1}\otimes\cdots\otimes m_{n-2}  \otimes x_{n-1}\right) \\
& = &m_{0}\otimes m_{1}\otimes...\otimes m_{n-2} \otimes m_{n-1}  \otimes x_{n},
\end{array}
\]
where $e(x_{n-1})=m_{n-1}\otimes x_n$.

For each sequence $\left(\chi_n \right)_{n\in \nn}$ associated with an $x\in X$, we define a corresponding sequence  $\left(\theta_{n}(x)\right)_{n\in\nn}$ in $G$ as follows. If $\chi_n = m_{0}\otimes m_{1}\otimes...\otimes m_{n-1}  \otimes x_{n}$, then we set 
\begin{equation}
\label{eq-theta}
    \theta_n(x) = m_{0}\otimes m_{1}\otimes...\otimes m_{n -1} \otimes z,
    \end{equation}
where $z \in \{T, L, R\}$ is fixed for all $n$.

 Note that there are several choices we have made while defining the sequence  $\left( \theta_n(x) \right)_{n\in\mathbb{N}}$. First we have to  choose representatives $m_0\otimes \cdots \otimes m_{n-1} \otimes x_n$ of $\chi_n$ and then we have also made a choice for $z \in \{ T, L, R\}$. However, we shall show that sequences $\left( \theta_n(x) \right)_{n\in\mathbb{N}}$ are Cauchy in $G$ and they have the same limit in $S$ for any given $x$ in $X$, independent of the choices we made.

We first show that the sequence $\left(\theta_n(x)\right)_{n\in\mathbb{N}}$ is a Cauchy sequence for every $x\in X$ and the choice of $N$ for any $\epsilon >0$ is independent of $x$.

\begin{lemma}
\label{cauchySeq}
For each $x\in X$, fix $m_0,m_1,\ldots \in M$ such that for all $n$, $\chi_n = m_0\otimes\ldots\otimes m_{n-1}\otimes x_n$ for some $x_n\in X$ and fix $z \in \{ T, L, R\}$ to define $\theta_n$. 
Then 
\begin{enumerate}
    \item 
    $\theta_n(x) = m_0\otimes\ldots\otimes m_{n-1}\otimes z$ for all $n$.
\item For every $\epsilon >0$, there is $N\in \mathbb{N}$ such that for all $x\in X$ and $p, q> N$, $d_G\left(\theta_p(x) , \theta_q(x) \right)<\epsilon$.     
\end{enumerate}
In particular, for every $x\in X$ the sequence $~(\theta_{n}(x))_{n\in\mathbb{N}}~$ is a Cauchy sequence in $G$.
\end{lemma}

\begin{proof}
Let $\epsilon>0\,$ and choose $\,N\in\mathbb{N}$ such that $\,2^{-N} <\epsilon$. For $\,q>p>N$ and $x\in X$, consider $$d_{G}(\theta_{p}(x),\theta_{q}(x))=d_{G}(m_{0}\otimes \cdots m_{p-1} \otimes z,m_{0}\otimes  \cdots \otimes m_{q-1} \otimes z  ).$$

The right hand side of this is equal to  
\[ d_{G}(m_{0}\otimes \cdots m_{p-1}\otimes l_{p}\otimes\cdots  \otimes l_{q-1}\otimes z',m_{0}\otimes  \cdots \otimes m_{q-1} \otimes z )\]
where $m_{0}\otimes \cdots m_{p-1}\otimes l_{p}\otimes\cdots  \otimes l_{q-1}\otimes z'\sim m_{0}\otimes  \cdots \otimes m_{p-1} \otimes z$ in $G$. It now follows from 
Remark~\ref{isometricembeddingintoG} 
that $$d_{G}(m_{0}\otimes \cdots m_{p-1}\otimes l_{p}\otimes\cdots  \otimes l_{q-1}\otimes z',m_{0}\otimes  \cdots \otimes m_{q-1} \otimes z  )$$ is equal to $$d_{M^q\otimes I}(m_{0}\otimes \cdots m_{p-1}\otimes l_{p}\otimes\cdots  \otimes l_{q-1}\otimes z',m_{0}\otimes  \cdots \otimes m_{q-1} \otimes z)$$ which is less than or equal to $\frac{1}{2^p}$ (by Lemma~\ref{totalybound}). But  $\frac{1}{2^p}<\frac{1}{2^N}<\epsilon$. Thus we have $d_{G}(\theta_{p}(x),\theta_{q}(x)  )<\epsilon$.

\end{proof}

One can now define a function $\,f:X\rightarrow S\,$ by setting  $$f(x)=\displaystyle\lim_{n\rightarrow \infty}\theta_{n}(x).$$
It is here that we explicitly use the fact that $S$
a complete space: $f$ is defined because the limits exist.
We shall now show that the definition of $f$ does not depend on the choices we have made.
Consider two choices of representatives for the sequence $(\chi_n)$, say we have $m_0,m_1,\ldots\in M$ and $m_0',m_1',\ldots\in M$ such that $$\chi_n=m_{0}\otimes m_{1}\otimes\cdots\otimes m_{n-1}  \otimes x_{n}=m_{0}^{'}\otimes m_{1}^{'}\otimes\cdots\otimes m_{n-1}^{'}  \otimes x_{n}'$$ for some $x_n,x_n'\in X$. Let $(\theta_n(x))_{n\in\mathbb{N}}$ and $(\theta_n'(x))_{n\in\mathbb{N}}$ be the two sequences corresponding to these choices, using $z$ and $z'$ in $\{T,L,R\}$ respectively. 
Write $\mathbf{m_n}=m_{0}\otimes m_{1}\otimes\cdots\otimes m_{n-1}$ and $\mathbf{m'_n}=m_{0}^{'}\otimes m_{1}^{'}\otimes\cdots\otimes m_{n-1}^{'}$. First note that the only way to achieve this equality is view the equivalences in (\ref{mustmake}), so we must have $z_n, z_n' \in \{ L, T, R \}$ such that  $\mathbf{m_n}  \otimes z_n=\mathbf{m_n'} \otimes z_{n}'$ in $G$. Thus, we have 
\begin{align*}
d_{G}(\theta_n(x),\theta_n^{'}(x)) & = d_{G}(\mathbf{m_n}\otimes z , \mathbf{m_n'}\otimes z')\\
&\leq d_{G}(\mathbf{m_n}\otimes z,\mathbf{m_n}  \otimes z_n) + d_{G}(\mathbf{m_n'} \otimes z_{n}', \mathbf{m_n'}\otimes z')\\
&\leq \frac{1}{2^{n}}+\frac{1}{2^{n}}=\frac{1}{2^{n-1}}
\end{align*}
by
Lemma~\ref{totalybound}.
Thus, we have $\displaystyle\lim_{n\rightarrow \infty}\theta_{n}(x) = \lim_{n\rightarrow \infty}\theta_{n}'(x)$, so the definition of $f$ does not depend on the choice of representatives of $\chi_n$ or the choice of $z\in \{T,L,R\}$. 

Hence, we have proved the following lemma.

\begin{lemma}
\label{indpend} Given $x\in X$, $\displaystyle\lim_{n\rightarrow \infty}\theta_{n}(x)$ is independent of the choice of representatives of $\left(\chi_n \right)_{n\in\mathbb{N}}$ and the choice $z\in \{ T, L, R\}$. 

In particular, $f:X\rightarrow S$ given by $f(x)=\displaystyle\lim_{n\rightarrow \infty}\theta_{n}(x)$ is well-defined.
\end{lemma}

\begin{remark}
\label{uniCauchy}
Note that the statement of Lemma \ref{cauchySeq} means that the sequence of functions $\theta_n : X\rightarrow G$ is uniformly Cauchy.
\end{remark}
 
We shall now show that this well-defined function $f:X\rightarrow S$ is the required mediating morphism to the final coalgebra. First note that $f$ preserves $\{T,L,R\}$: for $T_X$, 
\[ \chi_n = a\otimes \ldots\otimes a\otimes T_X,\] so $\theta_n(T_X) = a\otimes\ldots\otimes a\otimes z$ for some $z\in \{T,L,R\}$.  This sequence has the same limit as $(a\otimes\ldots\otimes a\otimes T)$ regardless of our choice of $z$, which is $T_G$.
Similarly for $L_G$ and $R_G $.  So we need to show that $f$ is continuous, and that it is the unique morphism in $\mc$ which makes the square in (\ref{comSquareMead}) commute. \rem{It is not too hard to show that $f$ must be the only map from $X$ to $S$ for which the square (\ref{comSquareMead}) commutes.}

\begin{equation}
\label{comSquareMead}
\begin{tikzcd}
X \arrow{r}{e} \arrow{d}{f} &  M\otimes X \arrow{d}{M\otimes f}\\
S \arrow{r}{s} & M\otimes S
\end{tikzcd}     
\end{equation}

\begin{lemma}
\label{defOfMediatingMor}
The function $f:X\rightarrow S$, given by $f(x)=\displaystyle\lim_{n\rightarrow \infty}\theta_{n}(x)$, makes the square in (\ref{comSquareMead}) commute. 

\end{lemma} 

\begin{proof}
We need to show $(M\otimes f)\circ e = s\circ f$.  First, let $m_0,m_1,\ldots\in M$ be such that for every $n\in\mathbb{N}$, $\chi_n = m_0\otimes\ldots\otimes m_{n-1}\otimes x_n$ for some $x_n\in X$, and choose $z\in \{T,L,R\}$ to define $\theta_n(x)$.  Observe that $e(x) = m_0\otimes x_1$ and for all $n$, 
\[ m_0\otimes \ldots\otimes m_{n-1}\otimes x_n = m_0\otimes (M^{n-2}\otimes e)\circ\ldots\circ e (x_1).\] 
So for all $n$, $(M^{n-1}\otimes e)\circ\ldots\circ e(x_1) = m_1\otimes\ldots\otimes m_n\otimes x_{n+1}$.  Thus,  $\theta_n(x_1) = m_1\otimes\ldots\otimes m_{n}\otimes z$.

So $f(x) = \displaystyle{\lim_{n\rightarrow\infty}} m_0\otimes\ldots\otimes m_{n-1}\otimes z$ and $f(x_1) = \displaystyle{\lim_{n\rightarrow\infty}}m_1\otimes\ldots\otimes m_{n-1}\otimes z$.  Then, by
Remark~\ref{remark-S},
\begin{align*}
    s \left(f(x)\right) &=s \left(\displaystyle\lim_{n\rightarrow \infty}m_{0}\otimes\cdots\otimes m_{n-1}  \otimes z\right)\\
    &=m_0 \otimes \displaystyle\lim_{n\rightarrow \infty}m_{1}\otimes\cdots \otimes m_{n-1} \otimes z\\
    &=m_0 \otimes \displaystyle\lim_{n\rightarrow \infty}m_{1}\otimes\cdots \otimes m_{n} \otimes z\\
    &= m_0\otimes f(x_1)\\
    & = (M\otimes f)(m_0\otimes x_1)\\
    & = (M\otimes f)\circ e(x).
\end{align*}
\end{proof} 

\begin{lemma}
\label{contOff}
$f:X\rightarrow S$ is continuous, and 
is therefore a morphism in $\mc$. 
\end{lemma} 

\begin{proof}
We will prove the continuity of $f$ at $x \in X$. Let $\,\epsilon >0\,$ be arbitrary. From Lemma \ref{cauchySeq}, there is  $N=N(\epsilon)$ such that for all $y\in X$ and $p\geq q> N$,  $d_S(\theta_p(y), \theta_q(y))<\epsilon/8$. By taking the limit as $p\rightarrow \infty$ we have $$d_S(f(y), \theta_q(y))<\dfrac{\epsilon}{8}$$ for all $y\in X$ and $q>N$.

Let $q$ be such that $q>N$ and $\frac{1}{2^q}<\frac{\epsilon}{4}$.
Since $q>N$, it follows that for any $y\in X$
\begin{align*}
    d_S(f(x),f(y)) &\leq d_S(f(x), \theta_q(x))+d_S(\theta_q(x), \theta_q(y))+d_S(\theta_q(y), f(y))\\
    &< d_S(\theta_q(x), \theta_q(y)) +\dfrac{\epsilon}{4}.
\end{align*} 

Let
\begin{equation}
    \label{eq}
    e^q=(M^{q-1}\otimes e)\circ \cdots \circ (M\otimes e)\circ e
\end{equation}
Since $e^q$ is continuous by Lemma~\ref{presFunctions}, there is $\delta=\delta(\epsilon) >0$ such that 
 $~d_{M^{q}\otimes X }(e^{q}(x),e^q(y))<\frac{\epsilon}{4}$ whenever $d(x,y)<\delta$.
 
Let $y\in X$ be such that $d(x,y)<\delta$. 
Note that, by the way we have defined $\theta_q$, by Lemma~\ref{totalybound} $$d(\theta_q(x),e^q(x)) < \frac{1}{2^q}<\frac{\epsilon}{4}\ \mathrm{and}\  d(\theta_q(y),e^q(y)) < \frac{1}{2^q}<\frac{\epsilon}{4}.$$ 

Thus, we get
\[
\begin{array}{rcl}
d_S(f(x),f(y)) & <& \frac{\epsilon}{4} + d_S(\theta_q(x),\theta_q(y))\\
& < & \frac{\epsilon}{4} + d_S(\theta_q(x),e^q(x)) + d_S(e^q(x),e^q(y)) + d_S(e^q(y),\theta_q(y))\\
& < & \frac{\epsilon}{4}+\frac{\epsilon}{4}+\frac{\epsilon}{4}+\frac{\epsilon}{4}\\
& = & \epsilon\\
\end{array}
\]
\end{proof}

\rem{
\begin{remark} In fact, $f$ is uniformly continuous.\footnote{Put in a little bit of explanation.  If we know that $e$ is (uniformly?) continuous then we can use the fact that $M\otimes-$ preserves uniform continuity? So $e^q$ is uniformly continuous? Jay: I agree.
LM: I might back off the claim that we preserve uniform
continuity.  I need to think about this.}
\end{remark}
}

\begin{lemma} 
\label{lemma-mustbe}
$f$ is the unique morphism in $\mc$ which makes the diagram in (\ref{comSquareMead}) commute.\end{lemma}
\begin{proof}
By forgetting the metric, the diagram in (\ref{comSquareMead}) is a diagram in $\set$, so if there are two morphisms in $\mc$ which make this diagram commute, these are morphisms in $\set$, so they must be equal since $(S,s)$ is the final coalgebra in $\set$. 
\end{proof}

Thus, we have the following: 

\begin{theorem}
$(S, s)$ is the final coalgebra for $F$ on  $\mc$.
\end{theorem}

We have seen that the Sierpinski Gasket, $\s$ with $\sigma: \s\rightarrow F\s$ is a final coalgebra in $\set$, however $\sigma$ is not a short map, so this is not the final coalgebra in $\ms$.  But we know that $S$ and $\s$ are bilipschitz equivalent, so in particular, there is a a 
Lipschitz bijection between them whose inverse is also Lipschitz.  Thus, there is an isomorphism in $\mc$ between $S$ and $\s$.

\begin{theorem} $(\s,\sigma)$ is the final $F$-coalgebra in $\mc$.\end{theorem}

\subsection{$\ms$}
It has been observed in \cite{Bhat} that $(S, s)$ is the final coalgebra for $F$ on $\ms$. Our main goal here is to see how this observation follows from the above discussion. We need to show that if $(X, e)$ is a coalgebra with $e$ being a short map, then the mediating morphism $f:X\rightarrow S$ is also a short map. This follows from the proof of Lemma \ref{contOff} with some minor modifications.

\begin{proposition}
\label{shortOff}
If $(X, e)$ is a coalgebra for $F$ on $\ms$ (i.e., $e$ is a short map), then the mediating morphism $~f:X\rightarrow S$ is also a short map.
\end{proposition} 

\begin{proof}
We modify the proof of Lemma \ref{contOff}.
Let $x,y\in X$. Let $\,\epsilon >0\,$ be arbitrary. Then there is some $q$ such that $2^{-q}<\frac{\epsilon}{3}$ and $ d_S(f(x),f(y)) < d_S(\theta_q(x), \theta_q(y)) +\frac{\epsilon}{3}$.
Note that $e^q$ from (\ref{eq})
is also a short map (by Lemma~\ref{presFunctions}) and hence  
\[~d_{M^{q}\otimes X }(e^{q}(x),e^q(y))<d(x,y).\]
 
Consider the two sequences $\left(\theta_q(x)\right)_{q\in\mathbb{N}}$ and $\left(\theta_q(y)\right)_{q\in\mathbb{N}}$. If $e^q(x)=m_{0}\otimes \cdots m_{q-1}\otimes {x_{q}}$ and $e^q(y)=n_{0}\otimes \cdots n_{q-1}\otimes {y_{q}}$, then $\displaystyle \theta_q(x)=  m_{0}\otimes \cdots m_{q-1}\otimes z$ and $\displaystyle \theta_q(y)= n_{0}\otimes \cdots n_{q-1}\otimes z$ for some $z\in \{T,L,R\}$. As in Proposition \ref{contOff}, we now have,
\begin{eqnarray*}
d_{S}(f(x),f(y)) & \leq  \dfrac{\epsilon}{3}~~  + & \frac{1}{2^{q}}+d_{M^{q}\otimes X}(e^{q}(x),e^{q}(y))+\frac{1}{2^{q}}\\ 
 & <  \dfrac{\epsilon}{3} ~~  + &\frac{\epsilon}{3}~~+~~d(x,y)~~+~~\frac{\epsilon}{3}~~=~~\epsilon~~+~~d(x,y).
\end{eqnarray*}

Since $\epsilon$ is arbitrary, $d_{S}(f(x),f(y))\leq d(x,y)$. Hence, $f$ is a short map. 	 
\end{proof}

\begin{theorem}[Theorem 16 of \cite{Bhat}]
$(S, s)$ is the final coalgebra for $F$ on  $\ms$.
\label{theorem-mshort}
\end{theorem}

However,  the unique coalgebra map from $(S, s)$ to $(\s, \sigma)$ is not short (only Lipschitz). Hence it follows that $(\s, \sigma)$ is not the final coalgebra for $F$ on $\ms$.

\section{Tripointed Metric Spaces with Lipschitz Maps}
\label{section-ml}

In this final section we turn our attention to $\ml$, the category of tripointed metric spaces with Lipschitz maps which preserve the $3$ designated points
$T$, $L$, and $R$.

In some ways, this is the most interesting case.  In $\mc$, we were able to obtain the Sierpinski gasket itself as the final coalgebra by showing that the final coalgebra $S$ in $\ms$ is a final coalgebra in $\mc$, and using the fact that $\s$ and $S$ are isomorphic in $\mc$.  In fact, they would be isomorphic in $\ml$, since they are bilipschitz equivalent, so the intention was to obtain $S$ as a final coalgebra in $\ml$.  

However, we find that not only is $S$ not a final coalgebra in this category, but there is no final $F$-coalgebra, or even weakly final $F$-coalgebra in this setting. 

\rem{
\begin{itemize}
    \item $G_\rho$ is the initial algebra again
    \item The case for $\ml$: in $\ms$ we don't get the Sierpinski Gasket, just bilipschitz equivalence to it for the final coalgebra.  So if we could get the same final coalgebra here, we would get the Sierpinski gasket.
    \item But this is even worse than $\mc$, we have no final coalgebra at all. 
\end{itemize}
}
\subsection{Initial Algebra}  

All of the work in Section~\ref{section-initial-algebra-continuous} goes
through.  We need only change
``continuous'' to ``Lipschitz'' throughout,
and thus change $\mc$ to $\ml$.
Again, what makes this work is the observation that 
finite spaces are bilipschitz isomorphic to 
discrete spaces.  Moreover,  
there is an adjunction
$D\dashv U$ between sets and metric spaces with Lipschitz maps.  

In this way, we see that the initial algebra
of $F$ on $\ml$ is the initial algebra in $\set$
endowed with the discrete metric.

\subsection{No Final Coalgebra} 

It turns out that there is no final $F$-coalgebra in $\ml$.  
Our proof employs the following coalgebra whose structure is reminiscent 
of the Cantor staircase function.
The carrier is 
 the tripointed space $$C = \{(x,0):x\in [0,1]\}\cup (\frac{1}{2},\frac{\sqrt{3}}{2})\}$$ whose distinguished elements are $T_{C} =(\frac{1}{2}, \frac{\sqrt{3}}{2})$, $L_C = (0,0)$, $R_C = (1,0)$ and whose metric is the Euclidean metric on $\mathbb{R}^2$. 

We need a family of coalgebra structure maps, one for each number $j\geq 4$.
Fix some such $j$.
Let $e:C\rightarrow F C$ be given by 
\[ e(x,y)= \begin{cases} 
      a\otimes T_{C} & y\neq 0 \\
      b\otimes L_{C} & y =0, x\in [0,\frac{1}{j}]\\
      b\otimes (g(x),0) & y=0, x\in [\frac{1}{j},\frac{2}{j}]\\
      b\otimes  R_{C} = c\otimes  L_{C} & y= 0, x\in [\frac{2}{j},1-\frac{2}{j}]\\
      c\otimes (h(x),0) & y = 0, x\in [1-\frac{2}{j},1-\frac{1}{j}]\\
      c\otimes R_{C} & y=0, x\in [1-\frac{1}{j},1]
   \end{cases}
\]
where $g(x) = jx-1$ and $h(x) = jx - (j-2)$. 
The graph of $e$ is pictured in
Figure~\ref{depiction}.

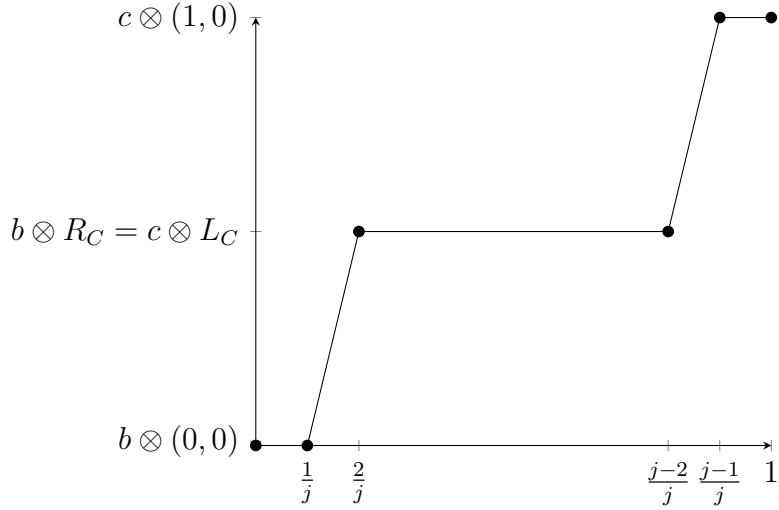
\begin{figure}[t]
\[
\begin{tikzpicture}
\tikzmath{
\r1 = 0;
\r2 = 1;
\r3 = 10;
\q1 = 2;
\q2 = 4;
\q3 = 16;
\q4 = 18;
\q5 = 20;
\x1 = \q1; \y1 =0;
\x2 = \q2; \y2=8;
\x3 = \q3; \y3=10;
\x4 = \q4;  \y4=16;
\x5 = \q5; \y5=18;
 } 
\begin{axis}[
axis x line=middle,
axis y line=middle,
xtick={0,\x1,\x2,\x3,\x4,\x5},
xticklabels={0,$\frac{1}{j}$, $\frac{2}{j}$, $\frac{j-2}{j}$, $\frac{j-1}{j}$,1},
xlabel near ticks,
ytick={.2,10,20},
yticklabels={$b\otimes{(0,0)}$,$b\otimes R_{C}=c\otimes L_{C}$,$c\otimes{(1,0)}$},
ylabel near ticks,
xmax=20,
ymax=20,
xmin=0,
ymin=0
]
\addplot[domain=0:\q1] {0};
\addplot[domain=\q1:\q2] {5*x-10};
\addplot[domain=\q2:\q3] {10};
\addplot[domain=\q3:\q4] {5*x-70};
\addplot[domain=\q4:\q5] {20};
\addplot[only marks,mark=*] coordinates{(0,0)(\q1,0)(\q2,10)(\q3,10)(\q4,20)(\q5,20)};
\end{axis}
\end{tikzpicture}
\]
\caption{This depicts the graph of $e$,
or rather its restriction to $A =\{(x,0):x\in [0,1]\}$.  On the $y$-axis, we show
$\{b,c\}\otimes A$.
On the $x$-axis, $j = 10$.
\label{depiction}}
\end{figure}

By a straightforward calculation, we can verify that this map is well-defined, and since $e$ is piecewise linear, it is easy to see that it is indeed Lipschitz with constant $\frac{j}{2}$.

Now we will use this to show that there is no final $F$-coalgebra in $\ml$.  The idea is that, given a coalgebra $(U,\mu:U\rightarrow FU)$, we can choose $j$ in the Cantor-like coalgebra based on the Lipschitz constant of $\mu$ such that it is impossible to find the required morphism $C\rightarrow U$, since its Lipschitz constant would have to be arbitrarily large.  So we are not only showing that there is no final coalgebra, but there is not even a weakly final coalgebra in this setting.  

\begin{proposition} There is no final $F$-coalgebra in $\ml$.\end{proposition}
\begin{proof}

Suppose there is a final $F$-coalgebra $(U,\mu:U\rightarrow F U)$.  Then $\mu$ is a Lipschitz isomorphism, so it is bilipchitz.  Let $u = \mu^{-1}$ and let $K\geq 1$ be such that $\frac{1}{K}d_{F U}(x,y) \leq d_U(u(x),u(y)) \leq K d_{F U}(x,y)$.

Let $(C,e:C\rightarrow F C)$ be the Cantor-like coalgebra described above with $j = 4K$ (so indeed, $j\geq 4$).  

For $n\geq 1$, let 
\[
\begin{array}{lcl}
x_n & = & (\frac{1}{j} + \ldots + (\frac{1}{j})^n,0)\\
y_n & = &  x_n + ((\frac{1}{j})^n,0)
\end{array}
\]
Then $d_C (x_n,y_n) = (\frac{1}{j})^n$, and observe that $x_n,y_n\in [\frac{1}{j},\frac{2}{j}]\times\{0\}$ 
for all $n\geq 1$. Further observe that $g(x_{n+1}) = x_{n}$ and $g(y_{n+1}) = y_{n}$ for all $n\geq 1$.  

Since $(U,\mu)$ is a final coalgebra, there is a unique 
$\ml$ morphism $f:C\rightarrow U$ such that $f = u\circ (F f)\circ e$.  Suppose $L>0$ is the Lipschitz constant for $f$.

We  show by induction on $n\geq 1$ that 
\begin{equation}
\label{showme}
d_U(f(x_n),f(y_n)) \geq \frac{1}{(2K)^n}
\end{equation}
Here is the verification for $n = 1$.

\[ f(x_1) = f((\frac{1}{j},0)) = u\circ (F f)\circ e((\frac{1}{j},0)) = u(b\otimes f(L_{C}))
= u(b\otimes L_U) .\]
Similarly,
$f(y_1) = u(b\otimes f(R_{C}))= u(b\otimes R_U)$.
So $d_U(f(x_1), f(y_1))$ is 
\[
 d_U(u(b\otimes L_U), u(b\otimes R_U)) \geq \frac{1}{K}d_{FU}(b\otimes L_U,b\otimes R_U) = \frac{1}{2K} d_U(L_U,R_U) = \frac{1}{2K}.
\]
We used Lemma~\ref{LemmaMetricOnTensor}.
This verifies (\ref{showme}) for $n=1$.
Assume (\ref{showme}) for some fixed $n$.
Observe first that 
\[
\begin{array}{lcll}
f(x_{n+1}) & = & (u\o Ff \o e)(x_{n+1}) &\mbox{since $f$ is a coalgebra morphism} \\
& = & (u\o Ff)(b\otimes g(x_{n+1})) & \mbox{by definition of $e$} \\
& = & (u\o Ff)(b\otimes x_{n}) & \mbox{since $g(x_{n+1})= x_n$} \\
& = & u(b\otimes f(x_n)) & \mbox{by functoriality of $F$} \\
\end{array}
\]
Similarly, $f(y_{n+1}) = u(b\otimes f(x_n))$.
Then using the induction hypothesis and the same reasoning as above, 
\[
\begin{array}{lcl}
d_U(f(x_{n+1}),f(y_{n+1})) & \geq & \frac{1}{K} d_{FU}(b\otimes f(x_n), b\otimes f(y_n)) \\
& = & \frac{1}{K} \frac{1}{2}d_U(f(x_n), f(y_n)) \\
& \geq & \frac{1}{2K} \cdot \frac{1}{(2K)^n} 
\end{array}
\]
Our induction step is complete.

We have established (\ref{showme}).  
Thus, we have \[
\frac{1}{(2K)^n} \leq d_U(f(x_n),f(y_n)) \leq Ld_{C}(x_n,y_n) = \frac{L}{j^n}.\]
Hence $\frac{j^n}{(2K)^n} \leq L$.
Recall that $j = 4K$.   Thus $2^n\leq L$.
Since this holds for all $n$, no such Lipschitz map $f$ can exist. 
Therefore $U$ is not a final $F-$coalgebra.
\end{proof}

\begin{corollary} $\SG$ is not a final coalgebra in $\ml$. \end{corollary}

\section{Conclusion}

The most important sections of this paper are Sections~\ref{section-mc} and~\ref{section-ml},
and a good way to understand those sections is as a search for the ``right'' categorical context
for our overall project.   To expand on this, let us go back to Freyd's characterization result
for the unit interval.  We mentioned the category $\BiP$ of bipointed sets 
and the relevant functor $K\colon\BiP\to\BiP$
in Example~\ref{ex-bipointed}.   Freyd introduced these to characterize 
the unit interval.  Here is a sketch of the proof, taken from~\cite{Bhat}.
Work in $\BiP$.
Let $U$ be the unit interval, and let $\upsilon\colon U \to KU$ be the coalgebra in 
Example~\ref{ex-bipointed}.
Suppose that we are given an arbitrary coalgebra $(A,\alpha\colon A\to KA)$.
We want to find a unique $h$ so that the diagram below commutes:
\begin{equation}
\label{eq-freyd}
\begin{tikzcd}
A  \arrow{r}{\alpha} \arrow{d}[swap]{h}&  KA \arrow{d}{Kh}\\
U & \arrow{l}{\upsilon^{-1}}   KU
\end{tikzcd} 
\end{equation}
It is easy to check that $\upsilon$ is a bijective isometry, so it is invertible.
Now $U$ is a complete metric space, and thus so is the function set $U^A$.
The subspace of functions which are morphisms in $\BiP$ is a closed subset (easily);
let us call this set $C$.
$C$  a complete metric space.   There is an endofunction on $C$
suggested by our diagram, namely $h\mapsto \upsilon^{-1} \o Kh \o \alpha$.
It is easy to check that this map is a contracting map.  So by the Contraction Mapping Theorem,
it has a unique fixed point.  Fixed points of our endofunction are exactly coalgebra morphisms
$(A,\alpha)\to (U,\upsilon)$.   And so we are done.

There are two noteworthy features of this proof.  First, even to prove a result about sets,
it was convenient to move to the setting of metric space, or even complete metric spaces.
This particular issue is central to this paper, because behind our discussion of
$\ms$, $\mc$, and $\ml$ lies the question of where we should be working to get final coalgebra results.
Second, the proof as we presented it is incomplete.   We ``used'' the Contracting Mapping Theorem
the way everyone thinks of it, but really one should state it correctly as 
``every contracting mapping of a \emph{non-empty} complete metric space has a unique fixed point.''
It is not immediate that $C \neq\emptyset$. However, for every bipointed set $A$, there is a 
morphism into $U$, namely $x\mapsto d(x,\bot)$.  So in this case, we are saved.

One would like to prove all of our results the same way: change the category in 
(\ref{eq-freyd}) from $\BiP$ to $\set$, and change the unit interval to the Sierpinski gasket $\s$.
This is a complete metric space.   We saw the algebra structure 
$\tau \colon F \s\to \s$ in Example~\ref{Sierpinski2}. 
So we replace $\upsilon$ in (\ref{eq-freyd})
with $\tau$.
However it is not at all obvious that every tripointed set $A$
has a continuous map $f\colon A\to\s$ preserving $T$, $L$, and $R$.  Indeed, showing this seems
as hard as proving the final coalgebra result.  Turning our argument around, the application
of the Contraction Mapping Theorem shows that finding \emph{any} such function $f$ 
is tantamount
to showing that there is a \emph{unique} one.  
(We could have proved Lemma~\ref{lemma-mustbe} using this type of argument.)
In a very real way, all of our work in Section~\ref{section-mc}
was devoted to just this task.  We feel that it would be difficult to find some $f$ directly,
mostly because the metric on $\s$ has features of Euclidean space that are not 
required of any tripointed metric space.  So instead of $(\s,\sigma)$, we worked with
$(S,s)$, where $S$ is the completion of the initial algebra
$(G,g)$, and $s$ is obtained from $g^{-1}$ using the Cauchy completion functor.
This map $s$ is an isometry, and so the coalgebra $(S,s)$ lives in 
$\ms$ as well as $\mc$ and $\ml$.  From the point of view of fractals, $G$ is the 
\emph{finite address space} with its intrinsic metric (unrelated to $\mathbb{R}^2$),
and $S$ is its completion, the \emph{infinite address space}.
Returning to $(S,s)$, the space is complete and the morphism is invertible, being an isometry
and a bijection.

Still, to follow our plan for the final coalgebra result,
we need to see that every 
object $A$, or every object $A$ which has a coalgebra structure $\alpha\colon A \to FA$,
 has a morphism $f\colon A \to S$.  
We also would like to know whether $\alpha$ being short (or Lipschitz, or continuous)
could guarantee that $f$ have the same property.  So here again we have this 
issue of what setting is best for our work.   We found in Section~\ref{section-final-coalgebra-mc}
that one can do all of this in the continuous setting.  Then we refined the observation to 
show that it works in the short setting as well.   So from this it seems that $\mc$ is the 
right setting for all of our work.  
(But also of interest is the fact that we obtained $f$ as a limit of maps $\theta_n$
(see~\ref{eq-theta}), and those were not necessarily continuous.)
However, the initial algebra does not work out the way we would want:
the metric is discrete for it.  So we conclude that perhaps there is no best setting for what we are doing.
Perhaps one needs to have all of the morphism classes in mind in this kind of work.

Returning for a moment to the Lipschitz setting, we have shown that it does not work out:
there simply is no final coalgebra in $\ml$.   This is not terribly surprising, since given a space
$B$ and a set $X$, the set of Lipschitz maps $g\colon X\to B$ is not in general a closed subpace of
the space of all maps.  So we would not expect our discussion above to work out in this case, 
and again it does not do so. 

The main next steps in this line of work would be to generalize all of this work to other fractals,
and to apply the results here.   There has been some work on generalizing to the Sierpinski carpet
(see~\cite{NoquezMoss}).   The work there is technically more complicated.  For that, one is 
``gluing along segments'', not just at discrete points.
Again, the issue of ``overall setting'' is hard because there are several choices for the overall
category.  Even with a choice in mind, 
the morphisms on the initial algebra chain 
(\ref{initialChain}) are not isometric embeddings, so the initial algebra is harder to work with.
It also seems possible to generalize what we did to work with fractals obtained by arbitrary 
arrangements of finite discrete ``gluing points,'' and this would be a good next step.
Finally, with the Sierpinski gasket itself, it would be nice to know whether the final coalgebra
results themselves could be used to simplify proofs of known facts, or to suggest new areas of
investigation.

\bibliographystyle{apalike}
\bibliography{gasket}
\end{document}